\documentclass[12pt,a4paper]{article}

\usepackage{amsmath}
\usepackage{amssymb}
\usepackage{amsxtra}
\usepackage{latexsym}
\usepackage{color}

\newtheorem{tm}{Theorem}

\newtheorem{cor}{Corollary}

\newtheorem{st}{Statement}
\newtheorem{ex}{Example}

\newtheorem{rem}{Remark}

\let\<\langle
\let\>\rangle

\numberwithin{equation}{section}  

\begin{document}

\title{On the traces of the $L_2$-solution of a general linear differential equation in the domain}

\author{
Vladimir P. Burskii
\footnotemark $^1$
}

\date{}

\maketitle

\begin{abstract}

This paper pertains to the general theory of boundary value problems for general linear differential equations with smooth coefficients in a bounded domain with a smooth boundary and contains new advances in the general theory related to the boundary properties of solutions. Specifically, conditions on the traces of a solution to a general differential equation on the boundary of the domain are found and studied, allowing the solution to be reconstructed from its traces and the right-hand side of the equation. For the case of a general equation with constant coefficients, the resulting conditions on the traces of the solution take the form of a generalized moment problem.

Bibliography: 22 titles.


\end{abstract}

\section{Introduction}

\footnotetext[1]{Moscow Institute of Physics and Technology, Institute of Applied Mathematics and Mechanics, {\tt burskii.vp@phystech.edu}}

It is well known that if some function $u\in H^l(\Omega)$ ($H^l(\Omega)=W_2^l(\Omega)$ is a Sobolev space) is a solution to a differential equation in a domain $\Omega\subset \Bbb R^n$ with smooth boundary $\partial\Omega$, then its boundary values or traces on the boundary are somehow related. Here, for an equation of order $l$, by the traces of a solution $u$ it is natural to understand a set of $l$ functions $u|_{\partial\Omega}$, $\partial_\nu u|_{\partial\Omega}$,..., $\partial^{(l-1)}_\nu u|_{\partial\Omega},$ where $\nu$ is the vector field of the outward normal in the neighborhood of the boundary.
For example, for the Poisson equation $\Delta u=f$, the functions $u|_{\partial\Omega}$ and $\partial_\nu u|_{\partial\Omega}$ cannot both be arbitrary, and specifying, for example, the first of them determines the second through the solution of the Dirichlet problem.
For the Poisson equation, the relationship between the solution traces is usually written as a Fredholm integral equation of the second kind, arising from representing the solution as a sum of three potentials. In this paper, from the perspective of the general theory of boundary value problems for partial differential equations, we consider the relationship between the solution traces of a general linear partial differential equation with $C^\infty$-smooth coefficients in a bounded domain $\Omega\subset \Bbb R^n$ with a $C^\infty$-smooth boundary $\partial\Omega$.

The fundamental idea behind the development of the general theory of boundary value problems for differential equations was John von Neumann's idea that a homogeneous boundary value problem can be understood as specifying the domain of some extension of a differential operator, initially defined on smooth compactly supported functions and acting on subspaces of the central space $L_2(\Omega)$.
The study of extensions of ordinary differential operators, as well as elliptic operators, led to the construction of a far-reaching and diverse theory with a large number of applications.
This same idea by J. von Neumann was used as the foundation for the general theory of boundary value problems for general partial differential equations in the work of M.Y. Vishik, where the boundary value problem is understood as specifying the domain of some extension of a minimal operator. Then, in L. H\"ormander's work, H\"ormander refined the concept of a boundary value problem, along with a proof of M.Y. Vishik's conditions for the existence of a well-posed boundary value problem for a scalar differential operation with constant coefficients.
At the same time, Ya.B. Lopatinsky ~\cite{Lop} found a condition for the Fredholm property of a general differential boundary value problem for an elliptic equation or system.
After the explosion of general interest in this topic in the 1960s and the recognition of the difficulties associated with the lack of serious research progress, analysts' interest in this area experienced a long decline. Among the works of the 1960s, we note the research of M.S. Agranovich \cite{Agr}, Yu.M. Berezanskii (see \cite{Ber}, pp. 93-115) and later A.A. Dezin \cite{Dez} on smoothly generated general boundary value problems. At present, the study of general formulations of boundary value problems is proceeding in the directions set in the mid-20th century by I.G. Petrovsky \cite{Pet}, when the very formulation of the boundary value problem was linked to the type of a specific differential equation, and among such formulations, a pressing issue is the selection of correct formulations of boundary value problems (see, for example, the book by A.V. Bitsadze \cite{Biz}, the monograph by A.P. Soldatov, beginning in \cite{Sold}). The author's book \cite{Bur1} is devoted to the development of a general theory of boundary value problems for general differential equations and systems in the domain. For a discussion of the representation of an equation's solution through a sum of potentials, see also the work \cite{Volov}. This work contains new advances in the general theory of boundary value problems related to the boundary properties of solutions. In particular, the behavior of the solution of a general differential equation on the boundary of a domain
is studied.

We obtain conditions on the solution's traces that allow this solution to be uniquely reconstructed from its traces and the right-hand side of the equation. For the case of a general equation with constant coefficients, the resulting conditions relating the solution's traces take the form of a generalized moment problem.

\section{Constructions of the general theory of boundary value problems} \label{S1}

Let $ \mathcal L =\mathcal L(x,D)= \sum\limits_{|\alpha |\leq l} a_\alpha (x) D^\alpha $ be a general differential operation, where $a_\alpha \in C^{\infty}(\bar\Omega)$ are complex-valued functions,
$ D^\alpha=\frac{(-i\partial)^{|\alpha|}}{\partial x^\alpha},$ and let $\Omega$ be an arbitrary bounded domain in $\Bbb R^n$. The operation $\mathcal L$ generates the formally adjoint operation
$\mathcal L^+ = \sum\limits_{|\alpha |\leq l} D^\alpha (\overline{a_\alpha (x)}\ \cdot )$, where $\overline{a_\alpha (x)}$ is the complex conjugate function.
Following the usual constructions of the general theory of boundary value problems, in the first subsections of this paper we will conduct our considerations in the space $H:=L_2(\Omega)$. This space will also be used as a central space in the sense of rigged spaces. Below we will use the standard notations $H^l(\Omega),$ $H_0^l(\Omega)$ for Sobolev spaces.

\quad {\bf Minimal operator} $L_0$, defined as the closure of the operator $\mathcal L$,
originally defined on $C_0^\infty(\Omega)$, in the graph norm
$\Vert u\Vert^2_L=\Vert u\Vert^2_H+ \Vert \mathcal L u\Vert^2_H$, generates a {\bf maximal operator} $L=(L_0)^*$ via the conjugation * in the Hilbert space $H$. The domains $D(L_0),$ $ D(L)$ of these operators are Hilbert spaces with the corresponding graph norm, and the subspaces $D(L_0)\subset D(L)$ and $D(L^+_0)\subset D(L^+)$ are closed.
We introduce the {\bf boundary space} $C(L)=D(L)/D(L_0)$ for the operator $L$, as well as the factor mapping $\Gamma:D(L)\to C(L)$. For the maximal operator $L$, we have a short exact sequence (\cite{Mac}, p. 22)
$$0 \to \ker L \to D(L) \to \text{Im}\,L \to 0 $$
and Im$\,L$ is closed in $H$. There is a similar sequence for the minimal operator, and, in addition, there are exact sequences
for the factorizations $D(L)/D(L_0)=C(L)$ and $\text{Im}\,L/\text{Im}\,L_0$.
Putting this together, we obtain the diagram

\begin{tabular}{ccccccccll}\label{D1}

 &     &0\ \ \ \ \ \          &                                     &           0\ \ \ \ &                                 &  0                                  &     & &     \\
 &     &$\downarrow\ \ \ \ \ $&                                     &$\downarrow$\ \ \ \ &                                 &$\downarrow $                        &     & &     \\
0&$\to$&$\ker L_0\ \ $        &$\stackrel{i_{L_0}}{\longrightarrow}$&$D(L_0)$            &$\stackrel{L_0}{\longrightarrow}$&$\text{Im}\ L_0$                     &$\to$&0&     \\
 &     &$\downarrow i_{\ker} $&                                     &$\downarrow i_0$    &                                 &$\downarrow\lefteqn{i_{\text{Im}}}$  &     & &     \\
0&$\to$&$\ker L  \ \ $        &$\stackrel{i_L}{\longrightarrow}$    &$D(L)  $            &$\stackrel{L}  {\longrightarrow}$&Im\ $L$                             &$\to$&0\ \qquad \qquad (D1)
&\\
 &     &$\downarrow\Gamma_{\ker}$&                                    &$\downarrow\Gamma$  &                                 &$\downarrow\lefteqn{\Gamma_{\text{Im}}}$&  & &     \\
0&$\to$&$C(\ker L)$           &$\stackrel{i_C}{\longrightarrow}$    &$C(L)$              &$\stackrel{L_C}{\longrightarrow}$&Im\ $L/$ Im\ $L_0$                   &$\to$&0&     \\
 &     &$\downarrow\ \ \ \ \ $&                                     &$\downarrow$\ \ \ \ &                                 &$\downarrow$                         &     & &     \\
 &     & 0\ \ \ \ \ \         &                                     & 0\ \ \ \           &                                 & 0                                   &     & &  \\
\end{tabular}
\par\noindent
where the operator
$\Gamma_{\ker} : \ker L \to C(\ker L):= \ker L/\ker L_0$ is the factorization map, and the operators $i_C$ and $L_C$ are defined by the formulas $i_C(u+\ker L_0)=u+D(L_0)$, $L_C(u+D(L_0))=Lu+\text{Im\,}(L_0)$ as the closing squares
of diagram (D1) up to commutative ones.

The commutativity of all squares is obvious. Thus, diagram (D1) is commutative,
all columns and the top two rows are exact. From the algebraic
3x3-lemma (see \cite{Mac}), we obtain the exactness of the bottom row. Proved

\begin{st}\label{St1}
The diagram \textup{(D1)}
is commutative, its rows and columns are exact.
\end{st}

Diagram (D1) clearly denotes the decomposition of the maximal operator into a direct sum $$L=L_0\oplus L_{C}$$ into the interior part $L_0$ and the boundary part $L_C.$

Consider the following Vishik conditions (see \cite{Vishik}):
\begin{equation}\label{Vish1} \text{operator}\ L_0: D(L_0)\to H
\ \text{has a continuous left inverse;} \end{equation}
\begin{equation}\label{Vish2}\text{operator}\ L_0^+: D(L_0^+)\to H
\ \text{has a continuous left inverse.}\end{equation}

Note that, as follows from standard arguments in functional analysis,
these statements are equivalent to the statements, respectively:

\begin{equation}\label{Vish3} \text{operator}\ L^+: D(L^+) \to H \ \text{is surjective;}\end{equation}
\begin{equation}\label{Vish4} \text{operator}\ L: D(L) \to H \ \text{is surjective.}\end{equation}

The condition \eqref{Vish1} is obviously equivalent to the inequality
\begin{equation}\label{Vish3}\exists C>0,\forall\varphi\in C^\infty_0(\Omega),\ \|\varphi\|_H\le C\| \cal L\varphi\|_H.
\end{equation}

\begin{ex}\label {Ex1}  

In the paper \cite{Bur1}, several classes of differential operators were indicated for which conditions (\ref{Vish1}), (\ref{Vish2}) are satisfied in a bounded domain.
This list includes

i) scalar operators with constant coefficients (L. H\"ormander \cite{Horm}),

ii) scalar operators of principal type with constant leading part
(see 
assertion \ref{ConstForce}),

iii) scalar operators of constant strength (see 
assertion \ref{MainType}),

iv) matrix operators with constant complex
coefficients with the Paneyakh-Fuglede property,

v) matrix operators that are uniformly elliptic in the Douglis-Nirenberg sense
in a domain with a smooth boundary. \end{ex}

Condition $(\ref{Vish1})$ is equivalent to ker\,$L_0=0$ and the subspace Im\,$L_0$ is closed in $H$, and condition $(\ref{Vish2})$ is equivalent to Im\,$L=H$. Furthermore, it is known that for a closed Im\,$L_0$, the orthogonal decomposition $H=$\ Im$\,L_0\oplus \ker\,L^+$ holds. Below, we will use the orthogonal projection operator $P_C:H\to\ker L^+$. Thus, under conditions (\ref{Vish1}),(\ref{Vish2}), diagram (D1) turns into diagram (D2):
\vskip 5 pt

\begin{tabular}{ccccccccll}\label{D2}
 &     &                      &                                     &           0\ \ \ \ &                                 &  0                                  &     & &     \\
 &     &                      &                                     &$\downarrow$\ \ \ \ &                                 &$\downarrow $                        &     & &     \\
 &     &0                     &$\longrightarrow$                    &$D(L_0)$            &$\stackrel{L_0}{\longrightarrow}$&$\text{Im}\ L_0$                     &$\to$&0&     \\
 &     &$\downarrow$          &                                     &$\downarrow i_0$    &                                 &$\downarrow\lefteqn{i_{\text{Im}}}$  &     & &     \\
0&$\to$&$\ker L      $        &$\stackrel{i_L}{\longrightarrow}$    &$D(L)  $            &$\stackrel{L}  {\longrightarrow}$& $H$                         &$\to$&0\ \ \ \ \quad\qquad \qquad (D2)
&\\
 &     &$\downarrow\lefteqn{\Gamma_{\ker}}$&                                    &$\downarrow\Gamma$  &                               &$\downarrow\lefteqn{\Gamma_{\text{Im}}}$&  & &     \\
0&$\to$&$C(\ker L)$           &$\stackrel{i_C}{\longrightarrow}$    &$C(L)$              &$\stackrel{L_C}{\longrightarrow}$&\ ker\ $L^+$                    &$\to$&0&     \\
 &     &$\downarrow$          &                                     &$\downarrow$\ \ \ \ &                                 &$\downarrow$                         &     & &     \\
 &     & 0                    &                                     & 0\ \ \ \           &                                 & 0                                   &     & &  \\
\end{tabular}

\begin{tm}\label{T1}
Diagram \textup{(D2)}, constructed under conditions \eqref{Vish1},\eqref{Vish2}
is commutative, and its rows and columns are exact.
\end{tm}
Diagram (D2) obviously denotes a decomposition of the maximal operator into a direct sum $$L=L_0\oplus L_{C}$$
into the interior part $L_0$ and the boundary part $L_C.$ This means, in particular,
that all questions concerning the boundary properties of functions from $D(L)$ or boundary value problems for the equation $Lu=f\in H$ pertain only to the operator $L_C$ and are not related to the behavior of the function $u$ inside the domain $\Omega$.

Similar to diagram (D2), a diagram can be constructed for the operators $L^+_0$ and $L^+$ with projections $\Gamma^+$ and $\Gamma^+_\text{Im}$.
Directly from diagram (D2) follows
\begin{st}\label{St1.3} (\cite{Vishik})
Let conditions (\ref{Vish1}),(\ref{Vish2}) be satisfied. Then
the decomposition
\begin{equation}\label{D(L)1} D(L)= D(L_0)\oplus\ker L\oplus W,
\end{equation}
where $W$ is a subspace of $D(L)$ such that $L:W\to \ker L^+$ is an isomorphism.
\end{st}

{\bf A homogeneous boundary value problem} (\cite{Horm}) is the problem of finding a solution to the relations
\begin{equation}\label{Luf} Lu=f,\ \ \Gamma u \in B,\end{equation}
\par
\noindent
where $B\,$ is some linear subspace in the boundary space $C(L)=D(L)/D(L_0)$,
defining the boundary value problem. Problem (\ref{Luf}) is called 
{\bf well-posed} if the operator $L_B=L|_{D(L_B)},$ $ D(L_B)= \Gamma^{-1} B$ is a solvable extension of the operator $L_0$,
i.e. if the operator $L_B: D(L_B)\to H^+ $ has a continuous inverse (which is also the right inverse of $L$).

The following statement is well known (M.Y. Vishik \cite{Vishik}; in L. H\"ormander's interpretation \cite{Horm}):

\begin{st}\label{StV} The operator $L_0 $ has a solvable extension
(and for the operator $L$ there is a well-posed boundary value problem (\ref{Luf}))
if and only if conditions (\ref{Vish1}) and (\ref{Vish2}) are both satisfied.
\end{st}

Directly from diagram (D2) follows
\begin{st}\label{St1.32} (\cite{Vishik})
Let conditions (\ref{Vish1}),(\ref{Vish2}) be satisfied.
For the extension $L_B$ to be solvable
(and for problem \eqref{Luf} to be well-posed), it is necessary and sufficient
that there exists a continuous operator $V:\ker L^+\to\ker L$ such that
\begin{equation}\label{D(L)}
D(L_B)=D(L_0)\oplus G(VL|_W),
\end{equation}
where $G(VL|_W)=\{w+VLw|w\in W\}$ is the graph
of the operator $VL|_W$. Moreover, $D(L)=D(L_B)\oplus \ker L$. \end{st}

Thus, example \ref{Ex1} presents known classes of operators for which there exists a well-posed boundary value problem in any bounded domain.
Since $\ker \mathcal L \vert_{C_0^\infty(\Omega} \subset \ker L_0$,
the equalities $\ker \mathcal L \vert_{C_0^\infty(\Omega)} = 0,$
$\ker \mathcal L^+\vert_{C_0^\infty(\Omega)} = 0$
are necessary for conditions \eqref{Vish1} and \eqref{Vish2} to be satisfied.
Theorem 13.6.15 from the book \cite{Horm4},
dating back to A. Plis, gives an example of an elliptic operator
for which condition \eqref{Vish1} is not satisfied, and therefore the following holds:

\begin{st}\label{StP} A fourth-order elliptic equation
$Lu=f$\ is constructed with $C^\infty$- coefficients in $\Bbb R^3$, for which in some ball there is not a single well-posed boundary value problem. \end{st}

Thus, example \ref{Ex1} presents known classes of operators for which there exists a
well-posed boundary value problem in any bounded domain.
Since $\ker \mathcal L \vert_{C_0^\infty(\Omega} \subset \ker L_0$,
the equalities $\ker \mathcal L \vert_{C_0^\infty(\Omega)} = 0,$
$\ker \mathcal L^+\vert_{C_0^\infty(\Omega)} = 0$
are necessary for conditions \eqref{Vish1} and \eqref{Vish2} to be satisfied.
Theorem 13.6.15 from the book \cite{Horm4},
dating back to A. Plis, gives an example of an elliptic operator
for which condition \eqref{Vish1} is not satisfied, and therefore the following holds:

\begin{st}\label{StP} A fourth-order elliptic equation
$Lu=f$\ is constructed with $C^\infty$- coefficients in $\Bbb R^3$, for which no well-posed boundary value problem exists in some ball. \end{st}

It is well known that for a uniformly elliptic equation
with a homogeneous real symbol in a bounded domain with a smooth ($C^{\infty}$) boundary $\partial\Omega$, the Dirichlet problem is well-posed (see, e.g., \cite{Lions}).

Recall also that the fulfillment of conditions \eqref{Vish1} and \eqref{Vish2} for an operator with constant complex coefficients in any bounded domain was proved by L. H\"ormander (\cite{Horm}).
Recall also that a differential operator with
coefficients $P^1(x,D)$ is weaker than the same operator $P^2(x,D): P^1\prec P^2$,
if for every $x\in\bar\Omega $ and for every $\xi\in{\Bbb R}^n\setminus 0,\ \ \tilde P^1(x,\xi)/\tilde P^2(x,\xi)\leq C$, where $(\tilde P(x,\xi))^2 =\sum\limits_{|\alpha|\leq l}|D_\xi^\alpha
P(x,\xi)|^2$. This comparison of operators means (see \cite{Horm}) that the inequality $\exists C>0,\forall\varphi\in C^\infty_0(\Omega),\ \|P^1\varphi\|_H\le C\|P^2\varphi\|_H,$ holds, and this is equivalent to the embedding $D(P^2_0)\subset D(P^1_0)$ with minimal operators.
Let us also recall that the differential operator $\mathcal L = \mathcal L(x,D)$ has a constant strength in the domain $\bar \Omega$ (\cite{Horm4}) if
$$\forall x\in\bar\Omega,\forall y\in\bar\Omega,\exists C>0,
\forall\xi\in
R^n,\ \tilde P(x,\xi)/\tilde P(y,\xi)\leq C,$$
where $|\tilde P(x,\xi)|^2 =\sum_{|\alpha|\leq l}
|D_\xi^\alpha \mathcal L(x,\xi)|^2$.

{\bf The constant strength operator} is represented as (\cite{Horm4})
\begin{equation}\label{P0sum}\mathcal L(x,D) = P^0(D) + \sum\limits_{j=1}^{N} q_j(x)P^j(D),\end{equation}
where $q_j\in C^\infty (\bar\Omega),$ and $ P^j $ are operators that form a basis
in the finite-dimensional vector space of operators
with constant coefficients that are weaker than
$P^0:P^j\prec P^0\,$. The formula \eqref{P0sum}, in particular, means
the coincidence of the domains of definition of the minimal operators
$D(L_0)=D(P^0_0)$. It is known (see, e.g., \cite{Horm4}, Theorem 13.5.2) that
for operators of constant strength, conditions (\ref{Vish1}) and (\ref{Vish2}) are satisfied in some neighborhood
of each point. The satisfaction of conditions (\ref{Vish1}) and (\ref{Vish2}) in the domain for operators of constant strength has also been proven for the case of coefficients $q_j$ that are analytic in some wider domain ((\cite{Horm4}), Section 13.5).

Conditions (\ref{Vish1}) and (\ref{Vish2}) in the domain are also valid for operators of real principal type of the form \eqref{P0sum}, that is, if $P^0(\xi)\in {\Bbb R}[\xi],\ \nabla P^0 \neq 0 $ for $\xi \neq 0$,
$q_i\in C^\infty(\overline\Omega)$,\,
ord\,$P^j<\,$ord\,$P^0$ (see \cite{Horm4}, section 13.5).

From the above and statement \ref{StV}, the following statements follow:

\begin{st}\label{ConstForce} The operator of constant force of the form \eqref{P0sum} admits a well-posed boundary value problem in some neighborhood of any point. If
its coefficients are analytic in the domain $\tilde\Omega$, then it admits
a well-posed boundary value problem in any subdomain $\Omega$ such that $\overline\Omega\subset \tilde\Omega$.\end{st}

\begin{st}\label{MainType} A real principal type operator of the form \eqref{P0sum}, where $P^j-$ are arbitrary operators of degree less than $m=\deg P^0=\deg \mathcal L$, admits a well-posed boundary value problem in a bounded domain $\Omega$.\end{st}

An example of a well-posed boundary value problem for a non-elliptic equation in a bounded domain is given by the following:

\begin{st}\label{St51}
\ (\cite{Bur1} or \cite{Bur5}) The following boundary value problem in the unit disk $D_{0,1}=\{x\in {\Bbb R^2}|\ |x|<1\}$
is well-posed in $L_2(D_{0,1})$:
\begin{equation}\label{5.1}
\square u=u_{x_1x_2}=f\in L_2(D_{0,1}),
u|_{\Gamma _1}=0,\ u_\nu ^{\prime }|_{\Gamma _2}=0,
\end{equation}
where $\Gamma _1=\{|x|=1,\ \frac {\pi}{2}\le \tau \le 2\pi \}$, $\Gamma _2=\{|x|=1,\ \pi \le \tau \le \frac{3\pi}{2}\}$, $\tau -$
angle variable.
\end{st}




\section{Traces of the $L_2$-solution on the boundary of the domain} \label{S3}

In this section, we will point out that, under broad assumptions on the differential operator, each solution of an equation in $L_2(\Omega)$ has some traces on the boundary of the domain, understood in terms of generalized functions, and we will consider the properties of the traces of solutions of a specific general equation.

Here and below, we assume that {a given bounded domain $\Omega$ has a $C^\infty$-smooth boundary $\partial \Omega.$}
For the general operator
$L=L(x,D)=$ $=\sum_{|\alpha |\le l}a_\alpha(x) D^\alpha$
we construct Green's formula for $H^l$-smooth $u$ and $v$:$\quad \forall u,\forall v\in H^l(\Omega):$
\begin{equation}\label{5.1}
\int\limits_\Omega (Lu\,
\overline{v}-u\,\overline{L^{+}v})\ dx=
\sum_{q=0}^{l-1}\ \int\limits_{\partial \Omega }\ L_{(l-q-1)}u\ \overline{\partial_\nu ^q v}\ ds.
\end{equation}
where for $j=0,...,l-1$\quad $L_{(j)}\,u-$ is a linear differential expression of order $j$. Expressions $L_{(j)}\,u$ are obtained for smooth functions $u$ and $v$ by rearranging the derivatives of $\int_\Omega Lu\,\overline {v}\,dx$ to $\int_\Omega u\,\overline {L^+v}\,dx$.
The same formula \eqref{5.1} can be written differently:
\vskip 2 pt
\centerline{$\int\limits_\Omega (Lu\,
\overline{v}-u\,\overline{L^{+}v})\ dx=
=-\sum\limits_{q=0}^{l-1}\ \int\limits_{
\partial \Omega }\ \partial_\nu ^q u\ \overline{L^+_{(l-q-1)}v}\ ds.$}

It can be calculated that for $k=0,1,...,l-1$ these expressions can be written as $L_{(k)}u =\sum\limits_{j=0}^{k} L_{(kj)} \partial^j_\nu u, $
$L^+_{(k)}v =\sum\limits_{j=0}^{k} L^+_{(kj)} \partial^j_\nu v $
with some operators
$L_{(kj)}:H^{l-j-1/2}(\partial\Omega) \to H^{l-k-1/2}(\partial\Omega),$
$\ L^+_{(kj)}:H^{l-j-1/2}(\partial\Omega) \to H^{l-k-1/2}(\partial\Omega)$,
Where $L_{(kj)}=0;\,L^+_{(kj)}=0$ for $j>k$, $0\leq k,j\leq l-1.$
Here $L_{(kj)},\,L^+_{(kj)}-$ are differential operators in the tangent
directions to the smooth boundary $\partial\Omega$ of order ord$\,L_{(kj)}$= ord$\,L^+_{(kj)}=k-j$, and the $C^\infty$-smooth coefficients of these linear operators are generated by the coefficients $a_\alpha(x)$ and the direction cosines of the normal $\nu$.
More explicit calculations of these expressions are available, for example, in the paper by \cite{Roitberg}, despite the fact that this paper is aimed at elliptic equations. For a given smooth function $u\in H^l(\Omega)$ {\bf differential expression
$L_{(q)}u, q=0,...,l-1$ will be called the $q$-th $L$-trace of the function $u$, and the entire set $\{L_{(q)}u\}|_{q=0}^{l-1}$ will be called the set of $L$-traces of the function $u$} (on the boundary $\partial\Omega$ of the domain $\Omega$).

If we are given some set of $l$ functions
$g_q\in H^{q-1/2}(\partial\Omega),$ $q=0,1,...,l-1,$ then
we can find a function $g\in H^l(\Omega)$ for which the set $\{g_q\} |_{q=0}^{l-1} $ is a set of ordinary traces
$g_q=\partial_\nu^q g|_{\partial\Omega}.$
Such a function $g$ can be obtained, for example, by finding the usual weak solution of the Dirichlet problem $\Delta^l g=0,$ $\partial_\nu^q\hat u=g_q|_{q=0}^{l-1}$ for the polyharmonic equation. Then the expressions $L_{(k)}g =\sum_{j=0}^{k} L_{(kj)} g_j,\ L^+_{(k)}g =\sum_{j=0}^{k} L^+_{(kj)} g_j $ will be called {\bf $L$-traces generated by the set of functions $\{g_q\}$}. It can be proved that the set of functions $g_j$ is quite sufficient for constructing the functions $L_{(k)}g$ and $L^+_{(k)}g$, since the expressions for
$L_{(k)}g$ and $L^+_{(k)}g$ include differential operators in tangent directions applied to the functions $g_q$. But since we do not specify the form of these expressions here, we specify the function in the domain that generates these expressions.


Let $D(\tilde L)$ denote the closure of $C^{\infty}(\overline\Omega)$ in $D(L)$ (or the completion in the graph norm) and, correspondingly, $D(\tilde L^+)$ denote the closure of $C^{\infty}(\overline\Omega)$ in $D(L^+)$. If some function $u\in D(L)$ is approximated
by smooth functions in the graph norm, then, passing to the limit in Green's formula \eqref{5.1}, we can see that each expression $L_{(j)}\,u$ will tend to a generalized function over $\partial\Omega$, which we will also denote by $L_{(j)}\,u$ and {\bf we will call the $j$-th $L$-trace of the function $u\in D(\tilde L)$}. In the paper \cite{Bur2} the following statement is proved (see also \cite{Bur5} or \cite{Bur1}):

\begin{st}\label{St31}
\qquad For any function $u\in D(\tilde L)$ there exist generalized functions
$L_{(j)}u\in $
\newline $H^{-j-1/2}(\partial\Omega),\ j=0,..., l-1$ such that Green's formula \eqref {5.1} holds for any function $v\in H^l(\Omega)$, where the integral, naturally, is understood to be the action of the functional on the underlying function. If the original function is smooth: $u\in H^l(\Omega)$, then the generalized function $L_{(j)}u\,$ is the $L$-trace of the function $u(x)$ introduced above and, moreover, $L_{(j)}u\in H^{l-j-1/2}(\partial\Omega)$.
\end{st}

It turns out that functions from $D(L_0)$ are distinguished by the fact that all these $L$-traces are equal to zero. In the same paper \cite{Bur2} the following statement is proved (see also \cite{Bur5} and \cite{Bur1}):
\begin{st}\label{St32}
An element $u\in D(\tilde L)$ belongs to the space $D(L_0)$ if and only if all its $L$-traces are trivial: $L_{(j)}\,u=0,$ $j=0,1,...,l-1.$
\end{st}
As examples show, the usual traces $u|_{\partial\Omega},\,\partial_\nu u|_{\partial\Omega},..., \partial^{(l-1)}_\nu u 
|_{\partial\Omega}$ for functions from $D(\tilde L)$, generally speaking, are not exist
even in generalized functions (\cite{Bur5},\cite{Bur1}). Therefore, the general linear boundary value problem understood in $L_2(\Omega)$ should be written in terms of $L$-traces:
\begin{st}\label{St33}
Under the condition $D(\tilde L)=D(L)$
the general linear boundary value problem $\Gamma u\in B$ for the equation $Lu=f$ in $L_2(\Omega)$ is written as
\begin{equation}\label{1.4}
B_\partial u:=B_i^{0}L_{(0)}u+B_i^{1}L_{(1)}u+...+B_i^{l-1}L_{(l-1)}u=0,
\quad i=1,...,l
\end{equation}
where $B_i^{k}:H^{-k-1/2}(\partial\Omega)\to H^{-i-1/2}(\partial\Omega)$
are given linear operators.
\end{st}

\begin{ex}\label{Ex2}
The conditions
\begin{equation}\label{1.42}
D(\tilde L)=D(L),\ D(\tilde L^+)=D(L^+)
\end{equation}
in a bounded domain are proved for the following classes of differential operators (see references and comments in the article \cite{Bur1} or in the book \cite{Bur5}):

\textcolor{blue}{\bf i)}\ \ scalar operators with constant coefficients
in a domain with the H\"ormander $T$ condition,

\textcolor{blue}{\bf ii)}\ \
for a scalar linear uniformly elliptic
operator in any domain with the cone condition,

\textcolor{blue}{\bf iii)} for a hypoelliptic operator with constant coefficients in any domain,

\textcolor{blue}{\bf iv)}
for any scalar differential operator of real
principal type in any domain compactly embedded
in the original domain.
\end{ex}




Consider the orthogonal decomposition $$H=\text{ Im}L_0\oplus \ker L^+,\quad f=f_0+f_C.$$
In the paper \cite{Bur1} (see also \cite{Bur5}), by applying the theory of rigged spaces to the triple $D(L_0)\subset H\subset D'(L_0)$, the following was obtained
\begin{st}\label{St52} Under conditions (\ref{Vish1}),(\ref{Vish2})

1) the operator ${\mathcal M}={\mathcal L}^+{\mathcal L}:D(L_0)\to D'(L_0)$ is a linear isomorphism, so the functions $f_0$ and $f_C$ can be obtained as follows:
$f_0=L_0{\mathcal M}^{-1}{\mathcal L^+}f,$ $f_C=f-f_0\in \ker L^+,$

2) For each solution $u\in D(L)$ of the equation $Lu=f$, there exists a unique
solution $u_C$ of the equation
$L_Cu_C=f_C, \text{ where } u=u_0+u_C,\ \ Lu_0=f_0.$
\end{st}

{\bf Explanation.} By virtue of the decomposition $L=L_0\oplus L_{C}$ (statement \ref{St1}),
the general boundary value problem
$Lu=f\in H,\ \Gamma u\in B$
under conditions (\ref{Vish1}),(\ref{Vish2}) reduces to the boundary value problem
$Lu_C=f_C\in \ker L^+,\ u_C=\Gamma u\in B,$
since the solution $u_0$ of the equation $Lu_0=f_0$ is uniquely found.

Now we will show that the equation $Lu=f$ can be transferred to the boundary space
$C(L)$ as a certain relation on the $L$-traces of the solution.

\begin{tm}\label{St54}
Let $\Omega$ be a bounded domain with a smooth boundary
$\partial \Omega $, and let the original differential operator $ \mathcal L = \sum\limits_{|\alpha |\leq l} a_\alpha(x) D^\alpha $ be an arbitrary linear operator with smooth complex coefficients satisfying the Vishik conditions (\ref{Vish1}), (\ref{Vish2}), the conditions \eqref{1.42}, and also the condition:

\qquad in the kernel $\ker L^+$ with topology $H$ the smooth functions are dense:
\begin{equation}\label{550} \textup{closure}_H (\ker L^+\cap H^l(\Omega))=
\ker L^+ . \end{equation}
Let $f\in H$ be a given function, which, by virtue of the orthogonal decomposition
$H=D(L_0)\oplus \ker L^+$, can be decomposed as $f=f_0+g,\ f_0\in D(L_0),
\ g=P_C\,f\in\ker L^+.$

{\bf 1.} For a function $u\in D(L)$ to be a solution to the equation ${\mathcal L}u=f\in H$, it is necessary that its $L$-traces $g_q=L_{(q)}u\in H^{-q-1/2}(\partial\Omega)$ and the function $g$ satisfy the identity (where the integral is understood as the pairing of the generalized and basic functions)
\begin{equation}\label{551}
\forall v\in \ker L^+\cap H^k(\Omega),\
\int\limits_\Omega g\,
\overline{v}
\,dx= \sum_{q=0}^{l-1}\ \int\limits_{\partial \Omega }\ g_{l-q-1}\ \overline{\partial_\nu ^q v}\ ds.
\end{equation}

{\bf 2.} Converse: Let $g_q\in H^{-q-1/2}(\partial\Omega)$ be the set of $L$-traces of some function $\hat u\in D(L):$ $g_q=L_{(q)}\hat u$, such that for some function $g\in\ker L^+$ the equality \eqref{551} holds. Then, for any function $f\in H$ such that $g=P_C\,f$, there exists a unique solution $u\in D(L)$ of the equation $Lu=f$ whose $L$-traces are the functions $g_q.$
\end{tm}
{\bf Proof 1}. Let $u\in D(L)$ be the solution of the equation ${
L}u=f\in H$. In the domain $\Omega$, conditions \eqref{1.42} are satisfied (see example \ref{Ex2}), therefore, by statement \ref{St31}, Green's formula \eqref{5.1} is valid for all
$v\in H^k(\Omega)$. Substituting the function
$f$ in place of $Lu$ and $v\in \ker L^+\cap H^k(\Omega)$ into formula \eqref{5.1}, we obtain the equality \eqref{551} due to the equality $\int_\Omega f\,\overline{v}\,dx=\int_\Omega f_C\,
\overline{v}\,dx.$

{\bf 2}. Conversely, let $f\in H,\ f=f_0+g,\,g\in\ker L^+,\,f_0=L_0u_0$ and functions $g_q\in H^{-q-1/2}(\partial\Omega)$ be given, which are $L$-traces of some function $\hat u\in D(L):\ g_q=L_{(q)}u$ and the identity \eqref{551} holds for them.
Then, by statement \ref{St31}, Green's formula \eqref{5.1} holds for $u=\hat u$ and for all $v\in H^k(\Omega)$, and for
$v\in \ker L^+\cap H^k(\Omega)$ we have the formula
\begin{equation}\label{552}\int\limits_\Omega L\hat u\,
\overline{v}
\,dx= \sum_{q=0}^{l-1}\ \int\limits_{%
\partial \Omega }\ g_{l-q-1}\ \overline{\partial_\nu ^q v}\ ds,
\end{equation}
Comparing with equality \eqref{551}, we obtain that
$\int_\Omega (g-L\hat u)\overline v\, dx=0$
for all $v\in \ker L^+\cap H^k(\Omega)$, and by \eqref{550}, that for all
$v\in \ker L^+$. From the orthogonality of the decomposition $H=$\ Im$\,L_0\oplus \ker\,L^+$
we obtain that $g-L\hat u\in$ Im$L_0,$ that is, there exists $\hat u_0\in D(L_0), \,L_0\hat u_0=g-L\hat u,$ which implies that $L(\hat u+\hat u_0)=g$. Then
the desired function $u\in D(L)$ is $u=\hat u+\hat u_0+u_0,$ since $Lu=g+f_0=f$, and from the statement \ref{St32} it follows that the $L$-traces of $u$ and $\hat u$ coincide. $\blacksquare$

Note that in the statement and proof of Theorem \ref{St54}, condition \eqref{550} can be eliminated if the right-hand side in equalities \eqref {551} and \eqref {552} is treated as a bilinear form
$(Lu\,\bar v\,dx-u\,\overline{L^+v})\,dx,$ without introducing the definition of $L$-traces.

Note also that in the orthogonal decomposition $f\in H,\ f=f_0+g,\,g\in\ker L^+,\,f_0=L_0u_0$, the orthogonal projection $Id_H-P_C:H\to$ Im$\,L_0$ can be constructed explicitly
according to the following scheme, presented in the author's works \cite{Bur1} or \cite{Bur6} (see also \cite{Bur4}). We will understand the operator $\mathcal L^+,$ as the adjoint operator to the minimal $L_0$ in the sense of the theory of rigged spaces, acting as in generalized functions $\mathcal L^+:H\to D'(L_0)$, where $D'(L_0)$ is the adjoint space of $D(L_0)$ with respect to the topology of the central space $H$. This operator gives rise to the generalized Dirichlet problem
$$\mathcal L^+\mathcal L u=f\in D'(L_0),\,u\in D(L_0)$$
(recall the statement \ref{St32}), which is posed as the problem of finding a solution $u\in D(L_0)$ of the integral identity \begin{equation}\label{eq52}
<L_0\,u,L_0\,\varphi>_H=<f,\varphi>,\,\forall\varphi\in D(L_0).
\end{equation}
It is proved that this Dirichlet problem is well-posed (there exists a continuous operator $M:D'(L_0)\to D(L_0)$ that solves the problem \eqref{eq52}) if and only if the first Vishik condition \eqref{Vish1} is satisfied. And now the projection $f_0=(Id_H-P_C)f$ can be constructed as follows: $f_0=L_0\,M\,\mathcal L^+f.$

Note that condition \eqref {551} characterizes the set of traces of solutions of the equation $Lu=f$ in the boundary space $C(L).$ However, the question of characterizing the boundary space $C(L)$ itself as a subspace in $H^{-l-1/2}(\partial\Omega)\times H^{-l-3/2}(\partial\Omega)\times ...\times H^{-1/2}(\partial\Omega).$ remains unresolved. Note that this question for the case of a uniformly elliptic operator $L=-\sum_{|\alpha|\le 2}a(x)D^\alpha$, $a_0(x)\le 0$ was solved by M.Y. Vishik in the same paper \cite{Vishik}. Namely, it was proved that in a bounded domain with a smooth boundary
\begin{equation}\label{eq521}C(\Delta)\approx H^{-1/2}(\partial\Omega)\times H^{1/2}(\partial\Omega).\end{equation}

More precisely, the boundary data spaces $H_1\{\psi_0|\,\exists u\in D(\Delta), \psi_0=u|_{\partial\Omega}\}$ and $G_1=\{\psi_1|\,\exists u\in D(\Delta), \psi_1=\partial_\nu u|_{\partial\Omega}\}$ introduced by Vishik in his paper \cite{Vishik} can be identified with the Sobolev-Slobodetskii spaces $H^{-1/2}(\partial\Omega)$ and $ H^{-3/2}(\partial\Omega)$ introdu\-ced by Slobodetskii after the publication of \cite{Vishik}. Moreover, in the work \cite{Vishik} it is proved that for each $\psi_0$ from the boundary space $H_1$, which we now understand as $H^{-1/2}(\partial\Omega)$, there exists a unique solution $u\in D(L)$ of the Dirichlet problem
$\Delta u=0,\, u|_{\partial\Omega}=\psi_0\in H^{-1/2}(\partial\Omega).$
Note that for the Laplace operator $L=\Delta$, the $\Delta$-traces are $\Delta_{(0)} u=u|_{\partial\Omega}=\psi_0\in H^{-1/2}(\partial\Omega)$, $-\Delta_{(1)} u=\partial_\nu u|_{\partial\Omega}=\psi_1\in H^{-3/2}(\partial\Omega)$, and that
due to uniform ellipticity, the leading symbol $\mathcal L(x,\xi)=\sum_{|\alpha|=2}a_\alpha (x)\xi^\alpha$ of the operator $L$ is separated from zero, and this, as can be shown, means that, given $L$-traces
$\chi_0=L_{(0)}u\in H^{-1/2}(\partial\Omega),\,\chi_1=L_{(1)}u\in H^{-3/2}(\partial\Omega)$, one can find the functions $\psi_0=u|_{\partial\Omega}\in H^{-1/2}(\partial\Omega),\,\psi_1=\partial_\nu u|_{\partial\Omega}\in H^{-3/2}(\partial\Omega)$.
In his work \cite{Vishik} introduced the operator $P:H_1\to G_1$, acting according to the rule $P\psi_0=\hat\psi_1,\, \hat\psi_1=\partial_\nu w|_{\partial\Omega},$ where $ \Delta w=0,\, w|_{\partial\Omega}=\psi_0.$ And it was proved that $\forall u\in D(L)$, the function $\partial_\nu u|_{\partial\Omega}-Pu|_{\partial\Omega}= \psi_1-P\psi_0=\psi$ lies in some space $H_2$, which we can identify with the space $H^{1/2}(\partial\Omega),$ and the operator $u\to\psi\in H_2$ is surjective, which proves the isomorphism \eqref{eq521}.


Now let the function $u\in H^l(\Omega).$ Interchanging $L$ and $L^+$, we obtain from formula \eqref {5.1} the identity
\begin{equation}\label{5.13}
\forall v\in D(L^+),\int\limits_\Omega (Lu\,
\overline{v}-u\,\overline{L^{+}v})\ dx=
-\sum_{q=0}^{l-1} \int\limits_{
\partial \Omega }\ \partial_\nu^q u\ \overline{L^+_{(l-q-1)}v}\ ds.
\end{equation}

As in theorem \ref{St54} for the equation $Lu=f$, this equality implies that the traces
$h_q=\partial_\nu^q u|_{\partial\Omega}$ of the solution $u\in H^l(\overline\Omega)$ satisfy the condition necessary for the equality
$Lu=f$,
\begin{equation}\label{5.14}
\forall v\in \ker L^+,\int\limits_\Omega f\,{\overline v}\ dx=
-\sum_{q=0}^{l-1} \int\limits_{\partial \Omega }\ h_q\ \overline{L^+_{(l-q-1)}v}\ ds,\end{equation}
in which the integral on the left can be written, as $\int_\Omega g\,\overline v\,dx,$
where $g=P_C f$, recall: $f=f_0+g,\,g\in\ker L^+,\,f_0\in\ $Im$\,L_0.$
Statement 1 of the following theorem is proved.

\begin{tm}\label{St55}
Under the conditions of Theorem \ref{St54}

{\bf 1.} Let $u\in H^l(\Omega)$ be a solution of the equation $Lu=f$
with some $f\in H$. Then the traces $h_q=\partial_\nu^q u \in H^{l-q-1/2} (\partial\Omega),$ $q=0,...,l-1$ of $u$ satisfy the identity \eqref{5.14}.

{\bf 2.} Let $f\in H$ and $h_q\in H^{l-q-1/2} (\partial\Omega),$ {\small $q=0,...,l-1$} be given, which satisfy the identity \eqref{5.14}.
Then there exists a unique function $u\in D(L)$ that is a solution of the equation $Lu=f$, whose $L$-traces
$L_{(q)}u\in H^{l-q-1/2} (\partial\Omega),\,q=0,...,l-1$ are generated by the functions $h_q$, that is, $L_{(q)}u=L_{(q)}h$.
\end{tm}
{\bf Proof}.
We prove assertion 2 of this theorem. We consider the question of the sufficiency of condition \eqref{5.14} for the existence of a solution $u\in D(L)$ of the equation $Lu=f$.
Given the set of traces $g_q\in H^{l-q-1/2} (\partial\Omega)$, we find a function
$\hat u\in H^l(\Omega)$ whose traces are $g_q=\partial_\nu^q\hat u$
(for example, as above, we find a weak solution to the Dirichlet problem
$\Delta^l \hat u=0,\,\partial_\nu^q\hat u|_{\partial\Omega}=g_q$).
Then, applying the arguments of the converse of Theorem \ref{St54}, we obtain the existence of a solution $u\in D(L),$ whose $L$-traces are obtained from the traces of $g_q$
as in formula \ref{5.1}. The proof is complete. $\blacksquare$

Note that such an $L$-trace $L_{(q)}u$, like the trace $g_q$, lies in the smooth space $H^{l-q-1/2} (\partial\Omega),$ but the solution $u\in D(L)$, whose existence and uniqueness is asserted, is, generally speaking, not smooth.
It can be proved that the smoothness of this solution $u\in D(L)$ can be obtained for an elliptic operator $L$, but we will not prove this here.

We also note the natural question of how to recover the solution whose existence is asserted in Theorems \ref{St54} and \ref{St55}.
A natural answer is to represent the solution as a sum of potentials,
when there is a fundamental solution $\mathcal E_y(x)$ of the operator of the original equation that has sufficiently smooth traces; see \cite{Bur5}, 1.3.4. If the function $\mathcal E_y(x)-$ is generalized, then it can be approximated by smooth functions, but the limiting step must be justified. Incidentally, the existence of a fundamental solution $\mathcal E^+_y(x)$ of the operator $\mathcal L^+$ makes it possible to explicitly solve the equation $L_0u=f_0\in$ Im$\,L_0$ (cf. the statement \ref{St52}) by means of the formula $$u(y)=<L_0u,\mathcal E^+_y>_x,$$ if we assume that the generalized function $\mathcal E^+_y(x)$ is approximated by smooth functions and use Green's formula.

\section{Traces of the solution of an equation with constant coefficients} \label{S4}

In this section, we will consider the case of a general differential equation with constant coefficients in the domain. Then the conditions for the connection of the solution traces \eqref{5.14} take the form of a generalized moment problem.

Let us recall an important definition. 
For a given differential operator $ \mathcal L = \sum\limits_{|\alpha |\leq l} a_\alpha D^\alpha $ with constant complex coefficients, the domain
$\Omega$ is called {\bf $\mathcal L-$convex for supports} if for every compact set $K\subset\Omega$, there is a compact set $K'\subset\Omega$ such that the conditions $\varphi\in C^\infty_0$ and
supp$\,\mathcal L\,\varphi\subset K$ imply that supp$\,\varphi\subset K'$
(\cite {Horm4}).

Below, we will use the following so-called approximation theorem (Malgrange, see \cite {Horm4}):
\begin{st}\label{St53}
Every solution $v\in H$ of the equation $Lv=0$ with constant complex coefficients in a $\mathcal L$-convex domain $\Omega$ can be approximated
(by linear combina\-tions)

\textcolor{blue}{\bf 1)} by polynomial-exponential solutions $Q(i(\xi,x))e^{i(\xi,x)},\ \xi\in \Lambda,$ where
$\Lambda=\{\xi\in{\Bbb C}^n\,|\ \sigma(\xi)=0\}$ is the algebraic variety of zeros of the symbol $\sigma(\xi)=$ $\sum_{|\alpha|\le l}a_\alpha\xi^\alpha $ in the general case,

\textcolor{blue}{\bf 2)} exponential solutions of the form $e^{i(\xi,x)},\ \xi\in \Lambda,$ if the symbol $\sigma(\xi)=$ $\sum_{|\alpha|\le l}a_\alpha\xi^\alpha$ of the original operator can be factored into a product of distinct irreducible polynomi\-als, then

\textcolor{blue}{\bf 3)} polynomial solutions $Q(i(\xi,x))$ if each factor of the symbol is zero at zero.
\end{st}

Note immediately that for the operator $L^+$, the symbol is the polynomial ${\overline \sigma}(\xi)= \sum_{|\alpha|\le l}\overline{a_\alpha}\xi^\alpha$ with complex conjugation. Let $\Lambda^+-$ denote the algebraic variety of zeros of the symbol $\overline{\sigma}(\xi)$. The approximation theorem allows us to simplify the conditions \eqref{5.14} on the traces of the solution and write them as a moment problem.

\begin{tm}\label{St56}
Let the original operator $ \mathcal L = \sum_{|\alpha |\leq l} a_\alpha D^\alpha $ be an arbitrary differential operator with constant complex coefficients, its symbol $\sum_{|\alpha|\le l}a_\alpha\xi^\alpha$ can be decomposed into a product of distinct irreducible polynomials, and the domain $\Omega$ is bounded by $\mathcal L$-convex for supports with smooth boundary $\partial \Omega $. For a function $u\in D(L)$ to be a solution of the equation $Lu=g\in \ker L^+$, it is necessary and sufficient that its $L$-traces $L_{(p)}u$ satisfy the relation
\begin{equation}\label{55}
\forall\xi \in \Lambda,
\int\limits_\Omega g(x)\,e^{-i(\xi,x)}dx=\sum_{q=0}^{l-1}\ \int\limits_{
\partial \Omega }\ L_{(m-q-1)}u\ \partial_\nu^q\,e^{-i(\xi,x)}\,ds_x,
\end{equation}
where, as above, $\Lambda$ is the algebraic variety of zeros of the symbol $\sigma$.
\end{tm}

{\bf Proof}
reduces to applying Theorem \ref{St54}, where as the function $v\in\ker L^+$
by virtue of statement \ref{St53} we can use the exponential solutions
$v=e^{i(x,\overline\xi)}$ of the equation $L^+v=0$, where $\xi \in \Lambda $ or
$\overline\xi\in\Lambda^+$.

\begin{cor}\label{1}
(For a more general situation, see \cite{Bur1}, \cite{Bur2}, or \cite{Bur5})
\newline
Under the conditions of Theorem \ref{St56}, for a function $u\in D(L)$ to be a solution of the equation $Lu=0$, it is necessary and sufficient that its
$L$-traces $L_{(p)}u$ satisfy the relation

\begin{equation}\label{555}
\forall\xi \in \Lambda,\sum_{q=0}^{l-1}\ \int\limits_{
\partial \Omega }\ L_{(m-q-1)}u\ \partial_\nu^q\,e^{-i(\xi,x)}\,ds_x=0.
\end{equation}
\end{cor}
The proof follows from Theorem 3 if we set $g=0.$

\begin{rem}\label{R1} Note that if you consider solutions of the equation $Lu=f\in H$,
where the operator satisfies conditions \eqref{Vish1},\ \eqref{Vish2} and \eqref{1.42},\ then, in accordance with the statement \ref{St1}, you must first expand the operator $$L=L_0\oplus L_{C}:D(L_0)\oplus C(L)\to \,\text{Im}\, L_0\oplus\ker L^+,$$ where you can assume that $C(L)=\ker L\oplus W\subset D(L),$ and $W\in D(L)$ is some closed subspace on which the restriction of the operator $L|_W$ realizes an isomorphism $L|_W:W\to\ker L^+$. Then note that by Condition \eqref{Vish1}, the operator $L_0$ has a continuous left inverse, which is simply its inverse if we assume that $L_0:D(L_0)\to\,$Im$L_0$, and the space Im$L_0$ is a closed subspace of $H$. And then note that the action of the operator $L_C$ is characterized by the statement of Theorems \ref{St54} or \ref{St56} in terms of $L$-traces.
\end{rem}

Let us consider the question of the well-posedness of the general linear boundary value problem
$\Gamma u\in B\subset C(L)$, which can be considered written in the form \eqref{1.4}, for the equation $Lu=f\in H$. As above, we understand this well-posedness as the existence of a solvable extension of $L_B$.

\begin{tm}\label{St58}
Let $\Gamma u\in B\subset C(L)$ be a well-posed boundary value problem for the equation $Lu=f\in H$, and let $u(x,\eta)$ be a function depending on $x\in\Omega$ and on some parameter $\eta\in \Bbb C^n$,
be a function $\forall\eta\in \Bbb C^n:$ smooth in $x$, $u(\cdot,\eta)\in H^l(\Omega)$, and be a solution to this boundary value problem for the equation $Lu(x,\eta)=e^{i(x,\eta)}\in\ker L^+$. Since the boundary value problem is well-posed, its solution $u(x,\eta)$ is unique and continuous in $C(L)$, depending on $\eta\in\Lambda^+$.
Then the traces $h_q(x,\eta)=\partial^q_\nu u(x,\eta)|_{\partial \Omega}$ of the function $u(x,\eta)\,$ for all $\eta\in \Lambda^+$ satisfy the identity
\begin{equation}\label{561}
\forall \xi \in \Lambda,
\int\limits_\Omega e^{i(\bar\eta-\bar\xi,x)}dx=-\sum_{q=0}^{l-1} \int\limits_{
\partial \Omega }\ h_q(x,\eta)\,L^+_{(l-q-1)}e^{i(\bar\xi,x)}\,ds_x\,, \end{equation}
where $\Lambda=\{\xi\in{\Bbb C}^n,\sum_{|\alpha|\le l}a_\alpha\xi^\alpha=0 \}$,
and for each $\eta\in\Lambda^+$, the set of functions $h_q$ satisfying the identity
\eqref{561} is unique. And the set of $L$-traces $L_{(q)}u(x,\eta)\in H^{l-q-1/2}(\partial \Omega)$ generated by the set of traces $h_q$ is continuous in $\eta$ as a mapping $\Lambda\to C(L)$ and, a fortiori, as a set of functions $\Lambda\to H^{-l+1/2}(\partial\Omega)\times H^{-l+3/2}(\partial\Omega)\times ...\times H^{-1/2}(\partial\Omega)$.
\end{tm}

{\bf Proof}.
In the proof of necessity in Theorem \ref{St56}, we substitute the projection $P_C\,f$ for the function $f\in H$ due to the orthogonality of
Im$L_0\perp \ker L^+$. Among the functions in the kernel $\ker L^+$, we select
exponential solutions, obtaining the assertion of the theorem. We use the fact that
the solution to a well-posed boundary value problem depends continuously on the right-hand side of $f\in H$.
$\blacksquare$

\begin{cor}\label{2}
Under the conditions of Theorem \ref{St56}, in order for the boundary value problem $\Gamma u\in B\subset C(L)$, which can be considered as written in the form \eqref{1.4}, to have at most a unique solution for the equation $Lu=f\in H$, it is necessary and sufficient that any of its solutions whose $L$-traces $L_{(p)}u$
satisfy the relation \eqref{555} be trivial: $u=0$.
\end{cor}

\begin{rem}\label{R2}
It seems useful to recall (statement \ref{St1.3}) that the boundary space $C(L)$ can be thought of as a subspace $\ker L\oplus W$ in $D(L)$, which
is the kernel of the operator $\mathcal L^+\mathcal L$ acting on distributions and bounded on $D(L)$, that is, $C(L)=\ker (\mathcal L^+ L)$.
For an operator with constant coefficients in a domain $\Omega-$ that is bounded and $\mathcal L$-convex for supports with smooth boundary $\partial \Omega $, conditions \eqref{1.42} are satisfied, and the polynomial-exponential solutions $Q_\eta(x)e^{i(\eta,x)}$ are dense in the space $C(L)=\ker (\mathcal L^+ L)$, where $\eta$ runs over the algebraic variety of zeros of the symbol of the operator $\mathcal L^+\mathcal L$.
Recall also (Statement \ref{St1.32}) that every well-posed boundary value problem generates a linear continuous (in the $H$ norm) operator $V:\ker L^+\to\ker L$ and vice versa; one can say that such operators parametrize the set of well-posed boundary value problems.
\end{rem}

We now restrict ourselves to operators with constant coefficients from the conditions of Theorem \ref{St56}. We assume that {\bf each function $u\in\ker L\subset H$ can be represented by an integral $u(x)=\int_\Lambda e^{i(\xi,x)}\, d\mu(\xi)$ with measure $d\mu(\xi)$ defined on the algebraic variety $\Lambda$
and denote the space of all such measures as $M(L,\Omega)$,} and similarly
define the space of measures $M(L^+,\Omega)$.
Each operator $V:\ker L^+\to\ker L$ generated by the boundary value problem $\Gamma u\in B\subset C(L)$ must map a function $e^{i(\eta,x)},\ \eta\in \Lambda^+$ from the kernel $\ker L^+$ to some function from the kernel $\ker L$, which we represent as an integral $v_\eta(x)=\int_\Lambda e^{i(\xi,x)}\, d\mu_\eta(\xi)$
with measure $d\mu_\eta(\xi)\in M(L,\Omega)$ defined on the algebraic variety $\Lambda$. The measure $d\mu_\eta(\xi)$ requires that the integral $v_\eta(x)$ lie in the space $L_2(\Omega)$ and be continuous in the $H$ norm
in $\eta\in \Lambda^+$.
Thus, we obtain a representation of each well-posed boundary value problem as a measure defined on the algebraic variety $\Lambda$ of zeros of the symbol of the operator $L$ and depending on the parameter $\eta\in \Lambda^+$. More precisely, solving the boundary value problem \eqref{1.4}
for the equation $Lu_{C,\eta}=e^{i(\eta,x)},$ we can, firstly, represent the solution $u_{C,\eta}$ of the equation as $u_{C,\eta}=u_{C,\eta}^0+u_{C,\eta}^1,$ where $u_{C,\eta}^1=Q(i(\eta,x))e^{i(\eta,x)}$ is the polynomial-exponential solution of the equation $Lu_{C,\eta}^1=e^{i(\eta,x)}$ (a particular solution), and $u_{C,\eta}^0=\int_\Lambda e^{i(\xi,x)}\, d\mu_\eta(\xi)$ is the general solution of the homogeneous equation $Lu_{C,\eta}^0=0.$ And now, substituting the function $u_{C,\eta}$ into the boundary value problem \eqref{1.4}, we obtain an integral equation for the unknown measure $\mu_\eta(\xi),$ representing the boundary value problem:
\begin{equation}\label{e61}\int_\Lambda B_\partial e^{i(\xi,x)}\, d\mu_\eta(\xi)= -B_\partial \big( Q(i(\eta,x))e^{i(\eta,x)}\big),
\end{equation}
and finding a solution to this integral equation means solving the original boundary value problem. In this case, the right-hand side $f_C$ of the equation $Lu_C=f_C$ is also represented by the integral $f_C=\int_\Lambda e^{i(\eta,x)}\, d\mu_f(\eta)$, therefore the solution $u_C$ is finally written as
$u_C(x)=\int_\Lambda u_{C,\eta}(x)\, d\mu_f(\eta)$. It is proved
\begin{tm}\label{St6}
Under the conditions of Theorem \ref{St56}, every well-posed boundary value problem $\Gamma u\in B\subset C(L)$, which can be considered as written in the form \eqref{1.4}, for the equation $Lu=f\in H$ generates a measure $d\mu_\eta(\xi)\in M(L,\Omega),$ depending on
$\eta\in\Lambda^+$ such that the function $u_{C,\eta}^0=\int_\Lambda e^{i(\xi,x)}\, d\mu_\eta(\xi)\in \ker L$ is continuous in $\eta\in\Lambda^+$ in the norm of $H$. This measure $d\mu_\eta(\xi)$ is a solution of the integral equation \eqref {e61} and allows us to write the solution of the boundary value problem \eqref{1.4} for the equation $Lu_C=f_C$ in the form $$u_C(x)=\int_{\Lambda^+} \left(\int_\Lambda e^{i(\xi,x)}\, d\mu_\eta(\xi)+Q(i(\eta,x))e^{i(\eta,x)}\right)\,d\mu_f(\eta),$$
where $u_{C,\eta}^1=Q(i(\eta,x))e^{i(\eta,x)}-$ is the solution of the equation $Lu_{C,\eta}^1=e^{i(\eta,x)}$, and the measure $d\mu_f(\eta)$ is taken from the expansion \ of \ the right-hand side \ $f_C=\int_{\Lambda^+}\,e^{i(\eta,x)}\, d\mu_f(\eta)$. \end{tm}

\section{Traces of the solution of a second-order equation with constant coefficients in a plane domain and the generalized moment problem} \label{S6}

In this section, we will demonstrate the relation between solution traces in a simpler case and point out that ill-posed boundary value problems can be important in applications not related to technology, where correct formulations are important, but, for example, in applications in other areas of mathematics.

If the operator symbol is a homogeneous polynomial, then each factor of the symbol is zero at zero, and the exponents in formulas \eqref{55},\eqref{555},\eqref{561} can be replaced by polynomials of the form $(i\xi,x)^k,\,\xi\in\Lambda,\,k\in\Bbb N\cup\, 0$ (a consequence of the approximation theorem, see statement \ref{St53}).

In particular, in a flat domain $\Omega$ (bounded, with a smooth boundary) for an equation (with arbitrary constant complex coefficients) of the form
\begin{equation}\label{60} Lu=A_{11}u_{xx}+2A_{12}u_{xy}+A_{22}u_{yy}=(a^1,\nabla)(a^2,\nabla)\,u=f\in H,
\end{equation}
Green's formula \eqref{5.1} on the flat curve $\partial\Omega$ can be written (see \cite{Bur6}) as a pair of conditions satisfied by the traces
of the solution $u\in H^2(\Omega)$:
\begin{equation}\label{61}
\int_{\partial \Omega} \left [u^\prime_{\nu_*} + \kappa
u^\prime_s \right ] (\tilde a^1,x)^N d s = \mu^1_N=\int_{\Omega}f(x)(\tilde a^1,x)^N\, dx,
\end{equation}
\begin{equation}\label{62}
\int_{\partial \Omega} \left [u^\prime_{\nu_*} - \kappa
u^\prime_s \right ] (\tilde a^2,x)^N d s = \mu^2_N=\int_{\Omega}f(x)(\tilde a^2,x)^N\, dx,
\end{equation}
where $s$ is a natural parameter on $\partial\Omega,$ $\kappa=\det A/2,$ $(\tilde a^k,a^k)=0,$ $\mu^1_N,\mu^2_N$ are given sequences defined by a given function $f$, $u^\prime_{\nu_*}$ is the conormal derivative, which is determined from the equality:
\vskip 5 pt
$\qquad\qquad\qquad\int\limits_\Omega (Lu\,
\overline{v}-u\,\overline{L^{+}v})\ dx=
\int\limits_{\partial \Omega }\ (u^\prime_{\nu_*}\, \overline{v}- u\,\overline{v^\prime_{\nu_*}}) \, ds,$
\par\noindent namely, $u'_{\nu_*} =
\sigma (\nu) u'_\nu - \frac 1{2 k} [ \sigma (\nu (s))
]^\prime_s \cdot u'_s\,;$ $\nu$ normal, $ \sigma(\xi)=(a^1,\xi)(a^2,\xi)$ operator symbol, $\ k=|\nu\,'_s|$ is the curvature of the curve $\partial \Omega$.

Note that the question of the existence of a solution to the homogeneous Dirichlet problem
$u|_{\partial \Omega}=0$ for equation \eqref{60}, using Corollary 2, becomes the question of the existence of a nontrivial function
$\alpha\in H^{1/2}(\partial\Omega)$ such that
\begin{equation}\label{63} \forall N\in {\Bbb N}\cup\,0,\ j=1,2,
\int_{\partial \Omega} \alpha(s)\, (\tilde a^j,x)^N d s = \mu^j_N= \int_{\Omega}f(x)\,(\tilde a^j,x)^N\, dx.
\end{equation}
The conditions written down have the form of a certain moment problem. In particular, if
we take the unit circle as the domain $\Omega$, and the vectors
$\tilde a^1=(1;i),\,\tilde a^2=(1;-i),$ then the equalities \eqref{63} turn into the classical trigonometric moment problem
\begin{equation}\label{64} \forall N\in {\Bbb N}\cup\,0,\ j=1,2,
\int_{\partial \Omega} \alpha(s)\, e^{iNs} d s = \mu^1_N,\\
\int_{\partial \Omega} \alpha(s)\, e^{-iNs} d s = \mu^2_N.
\end{equation}

Note that, for example, the question of the uniqueness of a solution $u\in H^{2}(\Omega)$ of the Dirichlet problem for equation \eqref{60}, using Corollary 2, turns into the question of the existence of a nontrivial function $\alpha\in H^{1/2}(\partial\Omega)$ such that
\begin{equation}\label{65} \forall N\in {\Bbb N}\cup\,0,\ j=1,2,
\int_{\partial \Omega} \alpha(s) (x,\tilde a^j)^N d s = 0,\quad
\end{equation}
which has the form of the indeterminacy problem of the moment problem \eqref{63}.

We note that the author's book \cite{Bur5} discusses in detail various questions
about boundary value problems for equation \eqref{60} and their connections with the moment problem
\eqref{61},\eqref{62}, in particular, questions about improving the smoothness of the solution. It should probably also be noted that the problem of uniqueness of the solution of the Dirichlet problem for the string vibration equation $u_{x_1x_2}=0\ $ in a domain with a so-called biquadratic boundary turned out to be closely related to several classical problems in geometry, algebra, and analysis (see \cite{Bur7}), namely, the Poncelet problem, the Pelle-Abel equation, the Toda chain, and several others, which made it possible to obtain a solution to each of these problems.

\end{document}